\documentclass[11pt,amsmath,amsthm,amsfonts,amssymb,amscd]{article}
\usepackage{epsfig}
\usepackage[latin1]{inputenc}
\usepackage{amsthm,amssymb}
\usepackage{amsmath}
\usepackage{graphicx}
\usepackage{psfrag}

\font \Bbbten=msbm10 \font \Bbbsev=msbm7 \font \Bbbfiv=msbm5
\newfam\Bbbfam \textfont\Bbbfam = \Bbbten
\scriptfont\Bbbfam=\Bbbsev \scriptscriptfont\Bbbfam=\Bbbfiv
\newcommand{\R}{\mathbb R}
\newcommand{\Z}{\mathbb Z}
\newcommand{\N}{\mathbb N}
\renewcommand{\epsilon}{\varepsilon}

\newcommand{\cE}{\mathcal{E}}

\newcommand{\dist}{{\rm dist}}

\newcommand{\diam}{{\rm diam} }

\newcommand{\cita}[7]{{\sc #1, }{\it #2, }{\small #3, {\bf #4 } (#5), p.
#6-#7.}}
\newcommand{\cit}[5]{{\sc #1, }{\it #2, }{\small #3, {\bf #4 } #5.}}

\newtheorem{maintheorem}{Theorem}

\newtheorem{Teo}{Theorem}[section]
\newtheorem{Lem}[Teo]{Lemma}
\newtheorem{Cor}[Teo]{Corollary}
\newtheorem{Prop}[Teo]{Proposition}
\newtheorem{Obs}[Teo]{Remark}

\theoremstyle{definition}
\newtheorem{Df}{Definition}[section]

\title{$N$-expansive homeomorphisms on surfaces.}
\author{A. Artigue, M. J. Pacifico, J. L. Vieitez
\thanks{M. J. Pacifico is partially supported by PRONEX - Dynamical Systems, CNPq, FAPERJ,  and 
Balzan Research Project of J. Palis,  Brazil,
A. Artigue partially supported by PEDECIBA,
J. L. Vieitez is partially supported by PEDECIBA and ANII, Uruguay.}
 }

\date{\today}

\begin{document}
\maketitle
\begin{abstract}
We exploit the techniques developed in \cite{Le} to study 
$N$-expansive homeomorphisms on surfaces.
We prove that when $f$ is a $2$-expansive homeomorphism defined on a compact boundaryless
surface $M$ with nonwandering set $\Omega(f)$ being the whole of $M$ then $f$ is expansive.
This condition on the nonwandering set cannot be relaxed: we present an example of a  $2$-expansive homeomorphisms on a surface with genus $2$  with   wandering points that is not expansive.
\end{abstract}\thanks{\footnotesize{
}}


\section{Introduction}
The notion of expansiveness was introduced in the middle of the twentieth century in \cite{Ut}.
 Expansiveness is a property shared by a large class of dynamical
systems exhibiting chaotic behavior. Roughly speaking, an expansive dynamical system
is one in which two different trajectories can be distinguished by an observer with an
instrument capable of distinguishing points at a distance greater than a certain constant
$\alpha>0$ (constant of expansiveness). Examples of expansive systems are for instance
diffeomorphisms acting hyperbolically in a $f$-invariant compact subset
of a manifold $M$.
This includes for instance Anosov systems or the non-wandering set of
Axiom A diffeomorphisms. Other examples are the pseudo-Anosov homeomorphisms in
surfaces of genus greater than $1$ \cite{HT}.

This notion is very important in the context of the theory of Dynamical Systems
and nowadays there is an extensive literature about these systems.
We recommend \cite{Le,Fa,Br,Mo2,PaVi,PPV,PPSV}
and references therein for more on this.

Recently  it was introduced in \cite{Mo1} the notion of $N$-expansiveness, generalizing the usual concept of expansiveness.
Roughly speaking, this corresponds to allow at most $n$ companion orbits   for a certain given and fixed positive integer $n$. For $n=1$ this notion coincides with the usual definition of
expansive. 

In this paper we study $N$-expansiveness on compact surfaces $M$.
We exploit the techniques developed in \cite{Le} to prove that
a $2$-expansive surface homeomorphism  with nonwandering set being the whole of $M$ is expansive.
We also prove that this condition on the nonwandering set can not be relaxed:
we exhibit a $2$-expansive homeomorphism on the two torus  whose nonwandering set is a proper subset of two torus that is not expansive.

To announce in a precise way our results let us recall the notions of expansiveness we
shall deal with.
To this end consider $f\colon X\to X$ a homeomorphism of a compact metric space\footnote{In this article we only consider the case of compact spaces.}
and define for $x\in X$ and $\epsilon>0$ the set
$$\Gamma_\epsilon(x,f)= \{y\in X:
d(f^n(x),f^n(y))\leq \epsilon,\, n\in\Z\}.$$
We will simply
write $\Gamma_\epsilon(x)$ instead of $\Gamma_\epsilon(x,f)$ when
it is understood  which $f$ we refer to.

\begin{Df} \label{expansivo}
 The homeomorphism $f$ is \emph{expansive} if there is $\alpha>0$ such that $\Gamma_\alpha(x)=\{x\}$ for all $x\in X$.
 Equivalently, given $x,y\in X$, $x\neq y$, there is $n\in\Z$ such that $\dist(f^n(x),f^n(y))>\alpha$.
\end{Df}

\begin{Df}[See \cite{Mo1}] \label{N-expansivo}
Given a positive integer $n$, the homeomorphism $f$ is $n$-\emph{expansive} if there is $\alpha>0$
such that $\sharp (\Gamma_\alpha(x))\leq n$ for all $x\in X$. Here $\sharp A$ stands for the cardinal of the set $A$.
That is, at most $n$ orbits $\alpha$-\emph{shadow} the orbit of $x$ by $f$.
\end{Df}

Clearly 1-expansiveness is equivalent to expansiveness.
Our results are the following:

%

\begin{maintheorem} \label{Teo A}
If $f\colon M\to M$ is a 2-expansive homeomorphism defined on a compact surface and $\Omega(f)=M$ then $f$ is expansive.
\end{maintheorem}

Using the classification of expansive homeomorphisms on surfaces given at \cite{Le,Hi} we
get the following

\begin{Cor} There is no $2$-expansive homeomorphisms on the sphere $S^2$
with nonwandering set the whole $S^2$.
A $2$-expansive homeomorphism on the torus with nonwandering set the whole torus is conjugated to an Anosov map
and on a surface with genus greater than $1$and the nonwandering set being the whole surface is conjugated to a pseudo Anosov map.

\end{Cor}


\begin{maintheorem}
 \label{Teo B}
There are $2$-expansive homeomorphisms of surfaces that are not expansive.
\end{maintheorem}

The paper is organized as follows. In Section \ref{s:n-expansive} we
consider  stable and unstable sets (see (\ref{d-stable}))  and following
\cite{Le}, we obtain that connected stable and unstable sets are arc connected.
Using these properties, we introduce the notion of a bi-asymptotic sectors
(see Definition\ref{Dfbias}). In Section \ref{s-nonbiasym} we prove that
a $2$-exapnsive surface homeomorphism with nonwandering set being
the whole of $M$ has no bi-asymptotic sectors. 
In Section \ref{Local product structure}, we consider $2$-expansive surface
homeomorphisms without wandering points, and
following \cite{le}, 
we prove that except for a finite set, every point of $M$ has local structure product.
In Section \ref{s-expansive} we prove our main result, Theorem \ref{Teo A}.
Finally, in Section \ref{sectionteoB} we present an example of a two-expansive
homeomorphism with wandering points on a bi-torus that is not expansive.

\section{Stable sets for $N$-expansive homeomorphisms of surfaces}
\label{s:n-expansive}

Let $M$ be a compact boundaryless surface and $f\colon M\to M$ a homeomorphism.
In this section we prove that $N$-expansiveness implies continuum-wise expansiveness
and obtain some usefull properties of stable and unstable arcs that will be used in the proofs.
For this, let us recall the notion of continuum-wise expansiveness \cite{Ka}.


\begin{Df} \label{cwexpansive}
A homeomorphism $f:M \to M$ is \emph{continuum-wise expansive} if there is
$\alpha > 0$ such that if $A$ is a nondegenerate subcontinuum of $M$, then there is
$n \in \Z$ such that $\diam(f^n(A)) > \alpha$, where $\diam(S) = \sup\{d(x, y):x,y \in S\}$
for any subset $S$ of $M$.
\end{Df}

\begin{Lem}
\label{l-continuum-wise}
Let $M$ be a compact metric space.
If $f\colon M\to M$ is $N$-expansive then $f$ is continuum-wise expansive
\end{Lem}
\begin{proof}
Let $A$ be a non degenerate sub-continuum of $M$ containing a point $x$.
Since $A$ is non degenerate there are infinitely many points on $A$, and since $f$ is $N$-expansive
there exist $y\in (A\backslash \Gamma_\alpha(x))$.
Since $f$ is $N$-expansive and $y\notin\Gamma_\alpha(x)$
we conclude that there is $n\in\Z$ such that
$\dist(f^n(x),f^n(y)>\alpha$.
In particular $\diam(f^n(A))>\alpha$. Thus $f$ is continuum-wise expansive.
\end{proof}

\begin{Obs}

When $f$ is a $C^1$ diffeomorphism defined on a compact manifold, we
 say that $f$ is robustly  $N$-expansive if there is a $C^1$ neighborhood
$\mathcal{V}$ of $f$ such
that all $g \in \mathcal{V}$ is also $N$-expansive.
Thus, Lemma \ref{l-continuum-wise} implies that a robust $N$-expansive diffeomorphism
is robust continuum-wise expansive and applying the results in
 \cite{Sa} we conclude that $f$ is a quasi-Anosov diffeomorphism. It is an Anosov diffeomorphism when $M$ is a surface.

\end{Obs}

Let $(X,d)$ be a compact metric space and $f:X\to X$ a homeomorphism. For $\epsilon>0$ we define the $\epsilon$-stable set of a point $x\in X$ wrt $f$ as
\begin{equation}\label{d-stable}
W^s_\epsilon(x,f)=\{y\in X\,:\, d(f^n(x),f^n(y))\leq \epsilon\, \forall \, n\geq 0\}\, .
\end{equation}
The $\epsilon$-unstable set of $x$ wrt $f$ is defined as $W^u_\epsilon(x,f)=W^s_\epsilon(x, f^{-1})$.
When there is no confusion we omit the reference to $f$ in the $\epsilon$- stable (unstable) sets. When it is not important to specify the value of $\epsilon$ we refer to these sets as local stable (resp. unstable) sets.

The following result is proved in \cite{Ka2, JRH}.
\begin{Prop} \label{continuos}
Let $f\colon M\to M$
be a continuum-wise expansive homeomorphism with a constant of expansivity $\alpha>0$, and $M$ a
compact boundaryless surface. Then
there is $\delta>0$ such that for any point $x\in M$ it holds that
$W^s_\epsilon(x)$ contains a non-trivial subcontinuum $D(x)$, such
that $x\in D(x)$ and $\diam(D(x))\geq\delta$.
Analogously, there is a non trivial subcontinuum $C(x)\subset W^u_\epsilon(x)$ with $x\in C(x)$ and $\diam(C(x))\geq\delta$.
In particular there are not Lyapunov stable points for $f:M\to M$.

\end{Prop}


We use these continua $C(x)$ and $D(x)$ given at Proposition \ref{continuos} to analyze the structure of $N$-expansive homeomorphisms of surfaces.

Given $x\in M$ and $\delta>0$ we let $B(x,\delta)=\{y\in M\,:\, d(x,y)<\delta\}$.
\begin{Prop} Let $M$ be a surface and $f:M\to M$ be a $N$-expansive homeomorphism and
let $0<\epsilon<\alpha/4$ be fixed.
For all $x\in M$, $C(x)\subset W^u_\epsilon(x)$ and $D(x)\subset W^s_\epsilon(x)$ are locally connected.
\end{Prop}
\begin{proof}
Arguing by contradiction assume that $C(x)$ is not locally connected.
Then for a fixed $x\in M$ we may choose $\delta>0$ such that $\delta<\alpha$  and in $B(x,\delta)$ we have that no point $\alpha$-shadows the orbit of $x$.
Let us restrict $C(x)$ to the connected component of $C(x)\cap B(x,\delta)$ containing $x$. We continue to call $C(x)$ to this connected component. If it were not true that $C(x)$ is not locally connected then we would have a neighborhood
 $V$ of $x$ such that there is a
sequence of continua $X_k$ converging in the Hausdorff metric to a continuum
 $X_{\infty}$ with $k\to\infty$ such that $\left(X_\infty\cup_{k\in\N} X_k\right)\subset C(x)\cap V$ and such that
 $X_k\cap X_{\infty}=\emptyset$ and $X_k\cap X_j=\emptyset$ for all $j,k\in\N$ and all $j\neq k$, see (see\cite[Chapter IV, \S 2]{Wi}). Choose a point $y\in X_k$ with $k>N$ and such that the distance between $X_{k}$ and its neighbors goes to zero with $k\to\infty$ (see \cite[Chapter IV]{Wi}). Observe that $X_k$ and also $X_\infty$ separate a neighborhood $U\subset V$ since $X_k$ is not connected but is part of a connected set as is $C(x)$. We may assume without loss of generality that $U=V$.
 Let $y_k\in X_k\subset C(x)$. Then $D(y_k)$ cannot intersect $X_\infty\bigcup\cup_j X_j$  because if 
 $D(y_k)$ has another point $z_k$ in common with $X_\infty\bigcup\cup_j X_j$, different from $y_k$, then $\dist(f^n(x),f^n(z_k))<3\epsilon<\alpha$.
 Thus $z_k$ $\alpha$-shadows the orbit of $x$. Since $k>N$ this leads to a contradiction.
 That means that $D(y_k)\cap V$ is in between $X_k$ and $X_{k+1}$ or in between $X_{k-1}$ and $X_k$. Letting $k\to\infty$ we find that  $X_\infty$ contains a non trivial sub-continuum $E$ such that is contained in both $W^s_\epsilon(x)$ and $W^u_\epsilon(x)$ contradicting continuum-wise expansiveness.
\end{proof}

Similarly we can prove that $D(x)$ is locally connected. It follows that $D(x)$ and $C(x)$ are both arc-connected,
see for instance \cite[Chapter Six, \S  II]{Ku}.
As $M$ is a compact surface there is $\delta'>0$ such that for any $x\in M$
  $B(x,\delta')$ is homeomorphic to a disk in $\R^2$.

Let $\delta>0$ be as in Proposition \ref{continuos} and assume, without loss, that $\delta<\delta'$  where $\delta'$ is as above.
Given $x\in M$ consider the family $\mathcal{A}^s=\mathcal{A}^s(x,\delta)$ of all arcs contained in $C(x)\subset W^s_\epsilon(x)$ with origin at
$x$ and endpoint at $\partial B(x,\delta)$. In a similar way we define $\mathcal{A}^u=\mathcal{A}^u(x,\delta)$.
 \begin{Lem}
If two arcs $\gamma,\gamma'\in \mathcal{A}^s$ meet at a point $y$ different from $x$ then they contain an arc through $x$ joining $x$ with $y$ contained in $\gamma\cap \gamma'$.
\end{Lem}
\begin{proof}
Indeed, if there is no such an arc then $\gamma\cap\gamma'$ is a disconnected set. Since $x\in\gamma$ and $x\in\gamma'$
we have that $\gamma\cup\gamma'$ is connected. Since $\R^2$ is a Janiszewski space (\cite[Volume 2]{Ku})
it holds that $\gamma\cup\gamma'$ separates $\R^2$, see \cite[Chapter Ten, \S 61]{Ku}.
Let $U$ be a bounded region of $\R^2\backslash (\gamma\cup\gamma')$.
By forward iteration we obtain that $\diam(f^n(\gamma\cup\gamma'))\to 0$ when $n\to\infty$.
Thus every point of $U$ is Lyapunov stable, contradicting that $f$ is continuum-wise expansive (see Proposition \ref{continuos}).
\end{proof}

Following Lewowicz we introduce the equivalence relation among the stable (unstable) arcs $\gamma \ni x$:
If $\gamma,\gamma'\in \mathcal{A}^s(x,\delta)$ we say that $\gamma\sim \gamma'$ if $\gamma\cap\gamma'$ strictly contains $x$.

\begin{Lem}
\label{finitapatas}
 For any point $x\in M$ there are finitely many equivalence classes of arcs $\gamma\in \mathcal{A}^s(x,\delta)$. Similarly for $\mathcal{A}^u(x,\delta)$.

\end{Lem}
\begin{proof}
Assume that there are infinitely many equivalence classes of arcs with origin in $x$ belonging to $W^s_\epsilon(x)$.
Let $\delta>0$ such that if $y\in B(x,\delta)\backslash \{x\}$ then the $f$-orbit of $y$ does not $\alpha$-shadow the orbit of $x$.
If for some $r_1>0$ there were infinitely many equivalence classes of arcs joining $x$ with $\partial B(x,r_1)$ then there will be a
subsequence $\{\beta_n\}$ of arcs, converging in the Hausdorff metric to a continuum $D_1$ contained in $W^s_\epsilon(x)$
joining $x$ with $\partial B(x,r_1)$. If there were infinitely many sub-continua $h_{n_k}\subset W^u_\epsilon(x)$ such that
$h_{n_k}$ is between $\beta_{n_{k-1}}$ and $\beta_{n_k}$ and joins $x$ with $\partial B(x,r_2)$ for some positive $r_1\geq r_2>0$ we have that also $h_{n_k}$ converges to a nontrivial subcontinuum $D_2\subset W^u_\epsilon(x)$. But $D_2\subset D_1$ and this contradicts continuum-wise expansiveness.
Thus, the subsequence $\{h_{n_k}\}$ cannot exist. That means that given $r_2>0$ there is $k_0$ such that if $k>k_0$ then no subcontinuum $h_{n_k}\subset W^u_\epsilon(x)$ and between $\beta_{n_{k-1}}$ and $\beta_{n_k}$ can intersect $\partial B(x,r_2)$. Observe that we may assume that $n_k-n_{k-1}\to\infty$ when $k\to\infty$.

Let $\gamma$ be a small arc in $ B(x,r_2)$ joining $\beta_{n_k}$ with $\beta_{n_{k-1}}$.
There is $\gamma $ such that for $z\in\gamma$ there is $C(z)\subset W^u_\epsilon(z)$
such that intersects either $\beta_{n_k}$ or $\beta_{n_{k-1}}$.
Otherwise we can find $h_{n_k}\subset W^u_\epsilon(x)$
between $\beta_{n_{k-1}}$ and $\beta_{n_k}$.
The subset of points of $\gamma $ such that $C(z)$ intersects $\beta_{n_k}$ is closed as
it is the subset of $\gamma$ of points $z$ such that $C(z)$ intersects $\beta_{n_{k-1}}$.
By connectedness of $\gamma$ there is $C(z)$ that cut both $\beta_{n_k}$ and $\beta_{n_{k-1}}$.
Therefore it cuts arbitrarily large number of $\beta_n$ contradicting $N$-expansiveness.
This proves that there are only a finite number of equivalence classes
\end{proof}

\subsection{Bi-asymptotic sectors}
\label{s-biasym}

For an $N$-expansive homeomorphism, $N\geq 2$, it is possible that a
local stable arc intersects twice a local unstable arc. In this case we introduce the following
\begin{Df}
\label{Dfbias}
A disc bounded by the union of a stable and an unstable arc is called
a \emph{bi-asymptotic sector}.
\end{Df}


Unlike in the expansive case, \cite[Lemma 3.2]{Le}, for a $N$-expansive homeomorphism $f$, $N\geq 2$, we cannot ensure that between two stable arcs in $\mathcal{A}^s(x,\delta)$ there is an unstable arc in $\mathcal{A}^u(x,\delta)$. We exhibit in section \ref{sectionteoB} a homeomorphism such that the mentioned property does not hold. Instead, for an $N$-expansive homeomorphism, we have the following alternative.

\begin{Lem}
\label{inestsepara}
Consider a small disc $\mathcal{D}\subset M$, a stable arc $\beta$ separating $\mathcal{D}$ and $x\in \beta$.
Denote by $U$ a component of $\mathcal{D}\setminus \beta$.
Then one of the following holds:
\begin{enumerate}
 \item there is an unstable arc in $U$ from $x$ to $\partial \mathcal{D}$  or
\item there is a bi-asymptotic sector in $U$.
\end{enumerate}
\end{Lem}
\begin{proof}
Assuming that item 1 does not hold let us prove item 2.
Since $f$ is $N$-expansive given $x\in M$ there is at most $N-1$ points shadowing the orbit of $x$.
Thus there is $r>0$ such that in $B(x,r)$ there is no point $y\neq x$,  such that its $f$-orbit $\alpha$-shadows the orbit of $x$.
Suppose that $x$ separates $\beta$ in $\beta_1$ and $\beta_2$.
We are assuming that there is no unstable arc
$\gamma\subset W^u_\epsilon(x)$, from $x$ to $\partial \mathcal{D}$, contained in $U$.
Let $\{S_k\}$ a family of arcs contained in $U\backslash \{x\}$, joining $\beta_1$ with $\beta_2$ and 
converging to $x$ when $k\to\infty$.
By Proposition \ref{continuos}, for any point $y\in S_k$ there exists a non trivial arc $D(y)\subset W^u_\epsilon(y)$ of diameter greater than a fixed  $\delta>0$.
If 
there 
is $y_k\in S_k$ such that $D(y_k)$ does not cut $\beta_1\cup\beta_2$ then $D(y_k)$ necessarily intersects
$\partial \mathcal{D}$ and taking limits, in the Hausdorff
distance on compact subsets, 
we obtain an unstable arc satisfying item 1 which is not possible since we are assuming that item 1 does not hold.

\begin{figure}[htb]
\begin{center}
\includegraphics{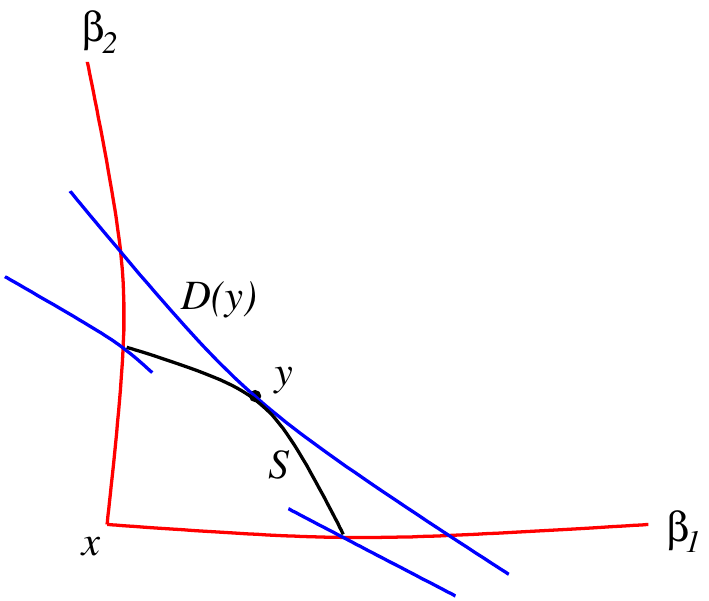}
\caption{The unstable arc $D(y)$ cuts both stable arcs $\beta_1$ and $\beta_2$.} \label{fig1}
\end{center}
\end{figure}

Thus, for every point $y\in S_k$
we have that $D(y)$ cuts either $\beta_1$ or $\beta_2$.
Let $S^{(1)}_k=\{y\in S_k\,:\, D(y)\cap\beta_1\neq \emptyset\}$ and $S^{(2)}_k=\{y\in S_k\, : \, D(y)\cap \beta_2\neq\emptyset\}$.
Then $S^{(1)}_k$ and $S^{(2)}_k$ are both closed and non-empty.
Since for all $k>0$ $S_k$ is connected we obtain that there is a point $y\in S_k$ such that $D(y)$ cuts
both $\beta_1$ and $\beta_2$ making a bi-asymptotic sector as can be seen at Figure \ref{fig1}.
%
%
%
%
\end{proof}

\section{Two-expansivess  and bi-asymptotic sectors}
\label{s-nonbiasym}

Let $M$ be a compact boundaryless surface.
In this section  $f:M\to M$ is a $2$-expansive homeomorphisms with nonwandering set $\Omega(f)$
being the whole of $M$.
We will prove that such homeomorphism has no bi-asymptotic
sectors (recall Definition \ref{Dfbias}) . To do so we proceed as follows.

Let $\alpha>0$ be an expansive constant for $f$, i.e., given any subset $C$ of $M$, if
$\diam( f^n(C))\leq\alpha$ for all $n\in\Z$ then $C$ has at most two points.

Let $D$ be a bi-asymptotic sector of diameter less than $\alpha$ bounded by a stable arc $a^s$ and an unstable arc $a^u$ as in Figure \ref{figbias1}.

\begin{figure}[h]
\begin{center}
\includegraphics{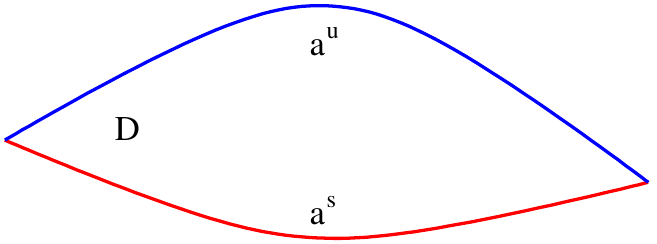}
\end{center}
\caption{Bi-asymptotic sector.}
\label{figbias1}
\end{figure}

For $p\in D$ define $C^s_D(p)$ and $C^u_D(p)$ as the
connected component of $W^s(p)\cap D$ and $W^u(p)\cap D$ containing $p$ respectively.

\begin{Lem}
 If $C^u_D(p)$ separates $D$ then it meets twice the stable boundary $a^s$ of $D$.
\end{Lem}

\begin{proof}
Observe that $D$ is a $2$-disk. Fix $p$ an interior point of $D$.
Since $C^u_D(p)$ separates $D$ we have that
$\partial D\cap C^u_D(p)$ has at least two points.
Moreover, since $C^u_D(p)$ is arc-connected, these two points can be joined by an arc $b$ contained in $C^u_D(p)$.
We need to show that these points are in $a^s$.
There are three possible cases.
In the first case $b$ cuts twice the unstable boundary of $D$ as in the first picture of Figure \ref{figbias2}.
\begin{figure}[h]
\begin{center}
\includegraphics{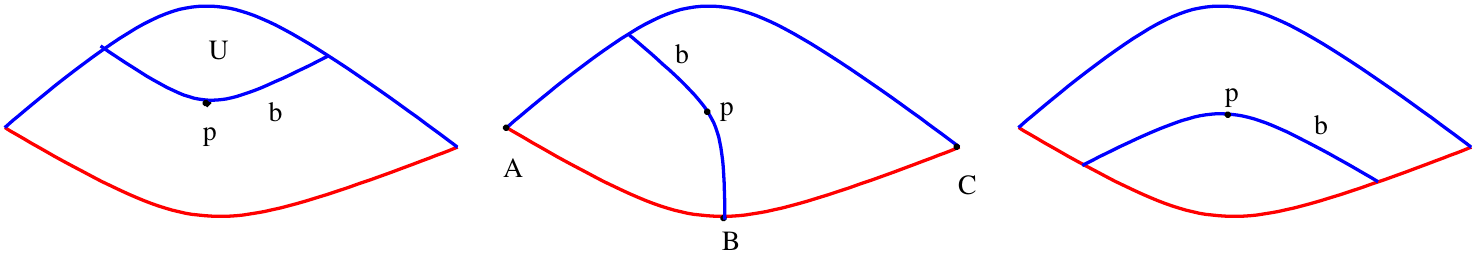}
\end{center}
\caption{The only possible case is the right hand side picture }
\label{figbias2}
\end{figure}
Both unstable arcs bound an open disc $U$, as in the figure.
This is a contradiction to Proposition \ref{continuos}, because the points in $U$ are Lyapunov stable.
The second case corresponds to $b$ intersecting the stable and the unstable arcs of the bi-asymptotic sector.
In this case we get  three points at $a^s$ contradicting the 2-expansiveness as shown in the second picture of Figure \ref{figbias2}: the points $A,B,C$ are in the same local stable and local unstable set.
Therefore the only possible case corresponds to the right hand side  picture at Figure \ref{figbias2}, that is exactly what we want to prove.
\end{proof}

In the set $\mathcal{F}^s=\{C^s_D(x):x\in D\}$ we can define an order as $C^s_D(x)<C^s_D(y)$
if $a^s$ and $C^s_D(y)$ are separated by $C^s_D(x)$. See Figure \ref{figorden}.
\begin{figure}[h]
\begin{center}
\includegraphics{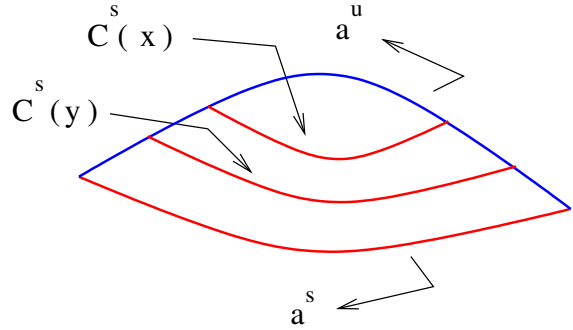}
\end{center}
\psfrag{s}{$a^s$}
\psfrag{u}{$a^u$}
\psfrag{x}{$C^s_D(x)$}
\psfrag{y}{$C^s_D(y)$}
\caption{Order of stable arcs separating a bi-asymptotic sector.}
\label{figorden}
\end{figure}

\begin{Lem}
\label{ordentotal}
 The order $<$ in $\mathcal{F}^s$  is a total order.
\end{Lem}

\begin{proof}
 Given $C^s_D(x),C^s_D(y)\in \mathcal{F}^s$, $C^s_D(x)\neq C^s_D(y)$, we have to prove that $C^s_D(x)<C^s_D(y)$ or $C^s_D(y)<C^s_D(x)$.
By contradiction assume this is not the case.
Therefore we can consider $\gamma_1,\gamma_2,\gamma_3\subset a^u$ three subarcs of the unstable
 boundary of the bi-asymptotic sector $D$. See Figure \ref{figbiasgammas}.
\begin{figure}[h]
\begin{center}
\includegraphics{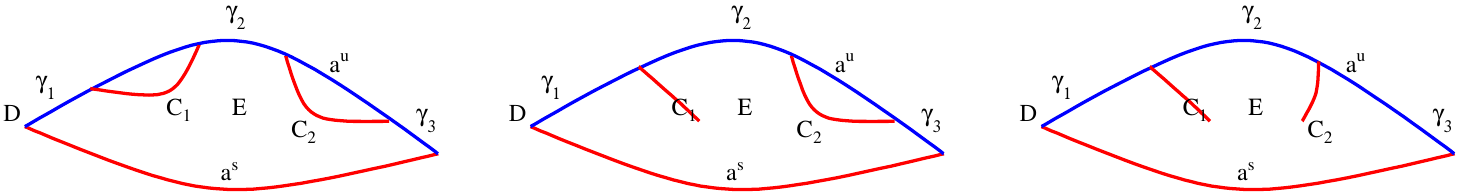}
\end{center}
\caption{Impossible cases for a 2-expansive homeomorphism.}
\label{figbiasgammas}
\end{figure}

Let $E$ be the connected component of $D\setminus (C^s_D(x)\cup C^s_D(y))$ containing $a^s$ as shown in Figure \ref{figbiasgammas}.
For $1\leq i<j\leq 3$, define
$$A_{ij}=\{x\in E:C^s_D(x)\cap \gamma_i\neq\emptyset,C^s_D(x)\cap\gamma_j\neq\emptyset\}.$$
We have that $C^s_D(x)\subset A_{12}$, $C^s_D(y)\subset A_{23}$ and $a^s\subset A_{13}$, so, these sets are not empty.
 It is easy to see that they are closed and by the previous lemma they cover $E$.
 Since $E$ is connected they can not be disjoint, but this contradicts 2-expansiveness.
 \end{proof}

Given a stable arc $b$ separating $D$ we
consider the map $g\colon b\to b$ defined by
\[
 C^u_D(x)\cap b=\{x,g(x)\}.
\]
 Notice that if $C^u_D(x)\cap C^s_D(x)=\{x\}$ then $g(x)=x$.
The hypothesis of 2-expansiveness implies that $C^u_D(x)\cap b$ has at most two points, therefore $g$ is well defined.

\begin{Lem}
\label{gcont}
For every stable arc $b\subset D$ separating $D$, the map $g\colon b\to b$ is continuous.
\end{Lem}

\begin{proof}
Since $b$ is homeomorphic to the interval $[0,1]$ we can consider in $b$ an order defining its topology.
We will show that $g$ is decreasing with respect to such an order on the arc $b$.
It is well known that this allows us to conclude that $g$ is continuous because $g\colon b\to b$ is bijective, in fact
$g=g^{-1}$ as can be easily seen from the definition of $g$.

By contradiction suppose that $g$ is not decreasing. Then there are $x,y\in b$ such that $x<y$ and $g(x)<g(y)$.
We have essentially two possible cases: $x<g(x)<y<g(y)$ or $x<y<g(x)<g(y)$. Other cases are obtained interchanging $x$ with $g(x)$
or $y$ with $g(y)$.
The first case contradicts Lemma \ref{ordentotal} because the arc from $x$ to $g(x)$ is not comparable with
the arc from $y$ to $g(y)$. The second case contradicts 2-expansiveness, because the unstable arc $\gamma_1$
from $x$ to $g(x)$ and the arc $\gamma_2$ from $y$ to $g(y)$ must have nontrivial intersection.
Then $\gamma=\gamma_1\cup\gamma_2$ is a unstable continuum containing
the four points $x,y,g(x),g(y)$.
Since these points are also in the stable arc $b$ we contradict 2-expansiveness.
\end{proof}

\begin{Lem}
\label{tangencia}
If $b\subset D$ is an unstable arc meeting twice $a^s$  then there is $z\in b$ such that
$b\cap C^s_D(z)=\{z\}$ (a fixed point of $g$).
\end{Lem}

\begin{proof}
We need to prove that there is a fixed point of $g$ in $b$ as in
Figure \ref{figbias3}.
\begin{figure}[h]
\begin{center}
\includegraphics[scale=0.70]{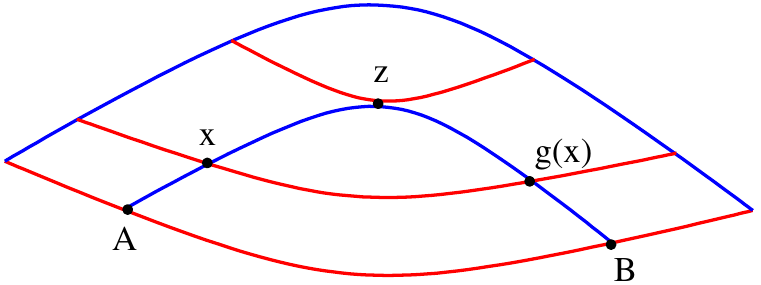}
\end{center}
\caption{}
\label{figbias3}
\end{figure}
Since $b$ is homeomorphic to an interval and $g$ is a homeomorphism reversing orientation
we have that $g$ must have a fixed point $z\in b$.
\end{proof}

\subsection{Regular bi-asymptotic sectors.}
\begin{Df}
A bi-asymptotic sector is \emph{regular} if for all $p$ interior to $D$ we have that
$C^s_D(p)$ and $C^u_D(p)$ separate $D$.
\end{Df}

\begin{Prop}
\label{regbiassect}
If $f$ is 2-expansive and $\Omega(f)=M$ then there are no regular bi-asymptotic sectors of diameter less than $\alpha$.
\end{Prop}

\begin{proof}
By contradiction assume that
$D$ is a regular bi-asymptotic sector of diameter smaller than $\alpha$.
Since there are no wandering points we have that there is $p\in D$ and $k>0$ arbitrarily large such that
$q=f^k(p)$ is in $D$.
Since the set of points $\xi$ such that $g(\xi)=\xi$ is of first category in the sense of Baire and $\Omega(f)=M$ there are (a residual subset of) points $p\in D$ such that $\{p,p'\}=C^s_D(p)\cap C^u_D(p)$ with $p\neq p'$.
The points $\{p,p'\}=C^s_D(p)\cap C^u_D(p)$ determine a regular bi-asymptotic sector $D_p$ contained in $D$
as in Figure \ref{figbias4}\,(b). 
\begin{figure}[h]
\begin{center}
\includegraphics[scale=0.65]{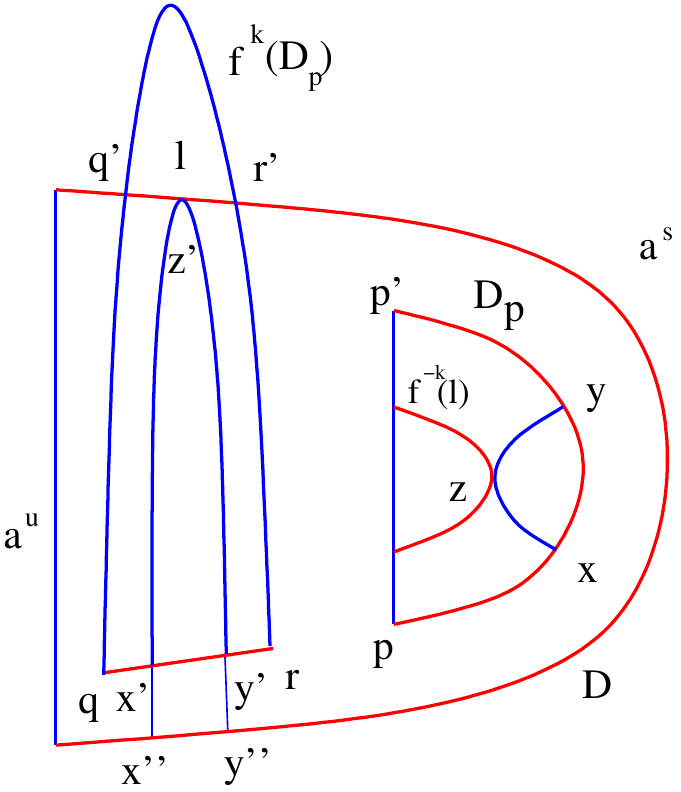}
\end{center}
\caption{ A regular bi-asymptotic sector\,.}
\label{figbias4}
\end{figure}
For arbitrarily large $k$, the stable
arc (in red at Figure~ \ref{figbias4}) defined by $p$ and $p'$ is transformed by $f^k$ into a stable arc with extreme points $q=f^k(p)$ and $r=f^k(p')$ that is contained in $D$.
The sector $D_p$ is transformed into $f^k(D_p)$ 
and the image by $f^k$ of the unstable arc $u(p,p')$ from $p$ to $p'$ is not contained in $D$.
Then, there are two points $q',r'\in a^s\cap f^k(u(p,p'))$ as in Figure \ref{figbias4}\,
(recall that $a^s$ is the stable arc in the bi-asymptotic sector $D$).
Now consider the stable arc $l=s(q',r')$ contained in $a^s$.
The stable arc $f^{-k}(l)$ separates the bi-asymptotic sector $D_p$ and
therefore we can apply Lemma \ref{tangencia} to obtain a point
$z\in f^{-k}(l)$ such that the unstable arc $u(z)$ through $z$ in $D_p$ meets $f^{-k}(l)$ only at $z$.
Take $x,y$ the intersection points of the stable arc in the boundary of $D_p$ and $u(z)$.
Consider the points $x'=f^k(x)$, $y'=f^k(y)$ and $z'=f^k(z)$ as in Figure \ref{figbias4}\,.
The unstable arcs in $D$ through $x'$ and $y'$ meet $a^s$ at $x''$ and $y''$ respectively.
The three points $z',x'',y''$ are in the intersection of a local stable arc
and a local unstable arc both contained in $D$ and so the orbits of  $x''$ and $y''$ $\alpha$-shadow that of $z$ and the orbit of $x''$ $\alpha$-shadows the orbit of $y''$, contradicting 2-expansiveness.
\end{proof}

\subsection{Bi-asymptotic sectors with spines}
Let $D$ be a bi-asymptotic sector with $\partial D=a^s\cup a^u$, where $a^s$ is a stable arc and $a^u$ is an unstable arc.

\begin{Df}
 A non trivial continuum $C^s_D(p)$  ($C^u_D(p)$) is a \emph{stable spine} (resp. \emph{unstable spine}) if it does not separate $D$.
\end{Df}

As before, we consider the map $g\colon a^u\to a^u$ defined by $a^u\cap C^s_D(x)=\{x,g(x)\}$.
Recall that, by Lemma \ref{gcont}, $g$ is continuous 
and reverses orientation.
As a consequence if a point $p\in a^u$ is in a stable spine then $p$ is a fixed point of $g$.

\begin{Lem}
Bi-asymptotic sectors contain at most one stable spine and one unstable spine.
\end{Lem}

\begin{proof}
Since $g$ is a homeomorphism of an arc and it reverses orientation we have that $g$ has exactly one fixed point.
So there is at most one stable spine. Similarly there is at most one unstable spine.
\end{proof}

\begin{Lem}
 If $\Omega(f)=M$ and $D$ is a bi-asymptotic sector then if there is a stable spine then there is an unstable one and it cuts the unstable spine in $D$.
\end{Lem}

\begin{proof}
 By contradiction suppose that there is a stable spine and an unstable spine and they are disjoint .
Denote by $S^s$ and $S^u$ the stable and the unstable spines respectively.
For all $x\in S^u$ we have that $C^s_D(x)$ separates $D$ because if this were not the case the spines meets at $x$.
We can define a partial order in $S^u$ as $x<y$ if $C^s_D(y)$ separates $x$ and $a^s$.
It is easy to see that there is a minimal $z\in S^u$ with respect to this order.
In this way we find a bi-asymptotic sector $D'\subset D$ bounded by a sub-arc of $a^u$ and an arc in $C^s_D(z)$.
Arguing in a similar way we find another bi-asymptotic sector $D''\subset D'$ without spines.
Then $D''$ is a regular sector, contradicting Proposition \ref{regbiassect}.

Now if there is a stable spine but no unstable one in a similar way we may find a regular bi-asymptotic sector leading again to a contradiction.
\end{proof}

\begin{Df}
We say that $y\in M$ has a \emph{local product structure} if there is a homeomorphism of $\R^2$ onto an open neighborhood of $y$ such that it maps
horizontal (vertical) lines onto open subsets of local stable (unstable) sets.
\end{Df}

\begin{Lem}
 \label{lps}
 Let $\gamma$ be a closed curve bounding a disc $U$ in the surface and consider $A;B,C,D$ four points in $\gamma$.
 Suppose that there is an open set $Q\subset U$ such that for all $p\in Q$ we have that there is a stable arc $C^s(p)$ and an
 unstable arc $C^u(p)$ meeting only at $p$ such that $C^s$ cuts $AB$ and $CD$ and $C^u$ cuts $BC$ and $DA$.
 Then there is a local product structure around every point of $Q$.
\end{Lem}

\begin{proof}
 Fix $p\in Q$. Take $L^s\subset C^s(p)$ a stable arc through $p$ and similarly $L^u$, with $L^s$ and $L^u$ contained in $Q$.
 Define $h\colon L^s\times L^u\to M$ as $h(x,y)=C^u(x)\cap C^s(y)$. This map is well defined, continuous and injective.
 By the Theorem of Invariance of Domain, we have that it is an open map.
 The map $h$ defines a local product structure around $p$.
\end{proof}

\begin{Prop}
 If $f$ is 2-expansive and $\Omega(f)=M$ then there are no bi-asymptotic sectors.
\end{Prop}

\begin{proof}
First notice that every regular stable leaf meets twice every unstable leaf.
That is because there are exactly one stable spine and one unstable spine in $D$.
Moreover, both cuts lie in different components of the complement in $D$ of the union of the stable with the unstable leaves.
By Lemma \ref{lps} there is a local product structure around every point in $D$ away from the spines.
Since we are assuming that there are not wandering points we conclude that periodic points are dense
in $D$ (arguing as for Anosov diffeomorphisms).

Take $p\in D$ a periodic point not in a spine.
Denote by $q$ the other point in the intersection of $C^s_D(p)$ with $C^u_s(D)$.
Given $\delta>0$ we can assume that $\dist(f^np,f^nq)\leq\delta$ for all $n\in\Z$.
Since $p$ is periodic we have $k$ such that $f^k(p)=p$. Obviously, $f^{n+jk}(p)=f^{n}(p)$ for all $n,j\in \Z$.
Therefore $$\dist(f^{n+jk}p,f^{n+jk}q)=\dist(f^{n}p,f^{n}(f^{jk}q))\leq\delta$$ for all $j,n\in\Z.$
Then, the points $p$ and $f^{jk}q$, with $j\in\Z$, contradicts the expansiveness of $f$ for the expansive constant $\delta$.
Since $\delta$ is arbitrary we conclude that bi-asymptotic sectors can not exist if $\Omega(f)=M$ and $f$ is 2-expansive.
\end{proof}

\section{Local product structure}

The goal of this section is to prove that except for a finite number of points,
every point $x\in M$ has a neighborhood with local product structure.
To do so, we follow closely \cite[Section 4]{Le}.

Given $x\in M$ and $\delta>0$, recall that $A^s_\delta(x)$ is the  family  of all arcs contained in 
$W^s_\epsilon(x)$ with origin at
$x$ and endpoint at $\partial B_\delta(x)$. 


Let $N^s_\delta(x)$ be the number of equivalence classes in $A^s_\delta(x)$.

\begin{Teo} \label{locprodest}
 If $N^s_\delta(x)\geq 2$ then there is a neighborhood of $x$ such that each $y\neq x$ in that neighborhood has a local product structure.
\end{Teo}

\begin{proof}
Let $a,b\in A^s_\delta(x)$ be two non-equivalent arcs.
Denote by $c$ an arc in $\partial B_\delta(x)$ connecting the end points of $a$ and $b$.
Assume that $c$ do not meet other stable arcs in $A^s_\delta(x)$.
Let $X$ be the sector determined by the arcs $a$ and $b$ with $c$ in its boundary
and denote by
$D$ an unstable arc in $A^u_\delta(x)$ separating $X$.

Let $c_1,c_2$ be arcs contained in $c$ so that $c_1$ begins at the endpoint
of $a$, $c_2$ ends at the endpoint of $b$ and $D\cap(c_1\cup c_2)=\emptyset$.
Let $V$ be an open connected neighborhood of $x$ in $X$ such that
for $y\in V$ the connected component of $C^s_{\delta/2}(y)\cap c$ that contains
$y$ is, in turn, included in $c_1\cup c_2$.
Moreover, we choose $V$ and $c_1,c_2$ such that the $\delta/2$-unstable set
through any point of $V$ does not meet $c$ in points that belong to $c_1\cup c_2$. Let $Q$ be the
subset of $V$ that consist of those $y$ satisfying the following conditions:
\begin{enumerate}
\item\label{condQ1} there is a stable arc $s(y)\subset B_\delta(x)$ that intersects $c_1$ and $c_2$,
\item\label{condQ2} there is an unstable arc $u(y)\subset B_\delta(x)$ that meets $c$ and $\partial B_\delta(x)\setminus c$.
\end{enumerate}

Let us show that $x\in Q$. Notice that condition (\ref{condQ1}) is satisfied by $x$.
In order to prove condition (\ref{condQ2}) notice that $a$ and $b$ are non-equivalent stable arcs and so they determine at least two sectors in $B_\delta(x)$.
Applying Lemma \ref{inestsepara} in both sectors we find an unstable arc satisfying condition (\ref{condQ2}).
Therefore $x\in Q$.

Now we show that $Q$ is open in $V$.
Let $y\in Q$ and consider the unstable arc $u(y)$ given by condition (\ref{condQ2}).
For $z\in V\cap u(y)$ we can consider $u(z)=u(y)$, an unstable arc through $z$ satisfying condition (\ref{condQ2}).
This arc separates $B_\sigma(x)$ in two sectors, so, applying Lemma \ref{inestsepara} on each sector, we have that
there is a stable arc $s(z)$ satisfying condition (\ref{condQ1}). Therefore $z\in Q$.
Similarly, for $t\in V\cap s(y)$ we find $u(t)$ satisfying (\ref{condQ2}), and $t\in Q$.
Consider the function that sends $(z,t)$ to the intersection point $h(z,t)=C^s_{\delta/2}(z)\cap C^u_{\delta/2}(t)$.
This intersection has at most one point by expansiveness, and is non-empty by definition of $Q$. So, $h$ is well defined.
It is easy to see that it is continuous and injective.
Therefore, by invariance of domain, it is an open map.
Then $y$ is in the image of $h$ and $y$ is an interior point of $Q$. This proves that $Q$ is open in $V$.
Lemma \ref{lps} implies $Q$ has a local product structure. Thus, $Q$ is a neighborhood of $x$ in the
sector $X$ with
local product structure.

To obtain
a local product structure for each $y\neq x$ close to $x$ we  notice that every point $y\neq x$ close to $x$ belongs
either to a sector bounded by stable arcs or to a sector bounded by unstable arcs. Repeating the above argument,
a finite number of times (recall Lemma \ref{finitapatas})
we prove that  $y$ admits a local product structure.
\end{proof}

\begin{Prop}[No spines] \label{sinespinas}
 For each $x\in M$ there is $\delta>0$ such that $N^s_\delta(x)\geq 2$.
\end{Prop}

\begin{proof}
By contradiction assume that there is $x\in M$ such that for all $\delta>0$ we have that 
there is only one equivalence class $A$ of stable arcs.
Pick a small $\delta$ and let $a\in A$ be a representative of that class.
Let $C\subset B_\delta(x)$ be the maximal stable continuum containing $x$.
If $C$ also contains points other than those in the arc $a$ we can join them to $x$, within $C$, because $C$ is arc-connected.
If all these arcs contain $a$ it is easy to see that for some smaller $\delta$
the stable set $C$ would consist of only one arc joining $x$ to $\partial B_\delta(x)$.

Assume then that there is a point $v\in C\setminus a$
that may be joined to $x$, within $C$, by an arc which does not contain $a$.
Thus, there is a point $u\in a$, $u\neq x$, and an arc $b\subset C$ with origin $u$ and endpoint $v$, whose intersection with $a$ is $\{u\}$.
Let $J$ be a Jordan curve through $x$ 
and $v$ such that $a$ and $b$ lie in its interior except for their endpoints.
Let $w\in a$, $w\neq x$ be the closest point to $x$ such that there is an arc $c\subset C$ with origin
$w$ and endpoint on $\widehat J$, $c\cap a=\{w\}$ where $\widehat J$ is a Jordan curve coinciding with $J$ except in a small neighborhood of $x$ and having $x$ in its interior. This point $w$ has to exist or we will have more than an equivalence class in $A(x,\sigma)$ for some $\sigma>0$.
Consequently, the arc contained in $a$, with origin $x$ and endpoint $w$, belongs (except for $w$)
to the interior of 
a sector bounded by local stable arcs (stable sector for short),
say $X$, defined as previously, with $w$ replacing $x$ and $J$ instead of $\partial B_\delta(x)$.
But on account of the local product structure on a neighborhood of $w$ in $X$, this implies that the stable set of $w$ meets twice some unstable set, which is absurd since there are no bi-asymptotic sectors.
Thus for some $\delta>0$, $C$ consists of an arc $a$ interior to $B_\delta(x)$ except for its endpoint at $\partial B_\delta(x)$.

Now, note that all interior points of the arc $a$ have a local product structure and therefore
their local unstable sets are transversal to $a$.

Let $\mathcal{U}$ be a small neighborhood of the arc $a$.

We can assume that for any point in the interior of the arc $a$ there is an unstable arc $\gamma$ transverse to $a$
and such that the end points 
 of $\gamma$ do not belong to $\mbox{clos}(\mathcal{U})$ where $\mbox{clos}(B)$ stands for the closure of a set $B $.

The intersection $\mathcal{U}\cap \partial B_\delta(x)$ is an arc which is separated by the endpoint $y$ of $a$ in two subarcs $c_1,\,c_2$ so that
  $\mathcal{U}\cap \partial B_\delta(x)\backslash \{y\}= c_1 \cup c_2\cup \{y\}$,
with $c_1\cap c_2=\emptyset$.

Take a small disk $\mathcal{V}$ of $x$ such that for all $y \in
\mathcal{V}$, the stable arc at $y$ is contained in $\mathcal{U}$.
Such a disk exists because otherwise we find, by a limit process, an arc not equivalent to $a$ contradicting our hypothesis.
The boundary of $\mathcal{V}$ is a circle $\mathcal{C}$. Let $\tau$ be the point of intersection of the arc $a$ with $\mathcal{C}$. The set
$\mathcal{C}\backslash \{\tau\}$ is a connected arc.

Define $\cE_1$ (resp. $\cE_2$) be the set of points of $t\in \mathcal{C}\backslash \{\tau\}$ such that its stable arc
 cuts $c_1$ (resp. $c_2$).
We have that $\cE_1$ and $\cE_2$ are
closed sets covering $\mathcal{C}\backslash \{\tau\}$.
If both $\cE_1$ and $\cE_2$ are non empty then since $\mathcal{C}\backslash \{\tau\}$ is connected and $c_1\cup c_2$ is not connected then there exists a local stable arc $S$ of a point at $\mathcal{C}\backslash \{\tau\}$ such that cuts both $c_1$ and $c_2$. But then there is an unstable arc $\gamma$ as above which cuts twice $S$ leading to a bi-asymptotic sector which is a contradiction.

So we can assume that only $\cE_1\neq \emptyset$. In this case we can find a stable arc cutting twice $\mathcal{C}\backslash \{\tau\}$ and choosing an arc $\gamma$ sufficiently near $x$ we find again a bi-asymptotic sector arriving again to a contradiction.
This finishes the proof.

\end{proof}

A point $x \in M$ without local product structure is called a {\em{ a singularity}} of $f$.

\begin{Cor}
\label{c-singularidade}
A $2$-expansive surface homeomorphism on $M$ with $\Omega(f)=M$ has a finite number of singularities.
\end{Cor}
\begin{proof}
The set of points where there is a local product structure is open. So the set of singularities is closed and
singularities are isolated by the previous propositions.
\end{proof}
%

\section{Expansiveness}
\label{s-expansive}

In this section we prove Theorem \ref{Teo A}, that is, we show that if $\Omega(f)=M$ and $f$ is 2-expansive then $f$ is in fact expansive.

First note that 
that every point $x\in M$ that is not a singularity has a neighborhood with local structure product.
We call such a neighborhood a {\em{box with local product structure} } that we refer to as a box, for short.

When $x$ is a singularity, from the proof of
Theorem \ref{locprodest} and Proposition \ref{sinespinas}
it follows that  $x$ has a neighborhood
 $\mathcal B$ such that its local stable (unstable) set consist of the union of $r\geq 3$ arcs
 that meet only at $x$.
Moreover, given any  unstable arc $\gamma^u$ in  $\mathcal B$ there is a stable arc $\gamma^s$
 through $x$  that intersects $\gamma^u$ only once, which implies that
$x$ is dynamically  isolated, that is,
 it coincides with the maximal invariant set in $\mathcal B$.
 We denote such a neighborhood  $\mathcal B$ a {\em{generalized box at $x$}}.
 Figure \ref{figSingPA} displays the main features of a generalized box.

\begin{figure}[h]
\begin{center}
\includegraphics[scale=0.50]{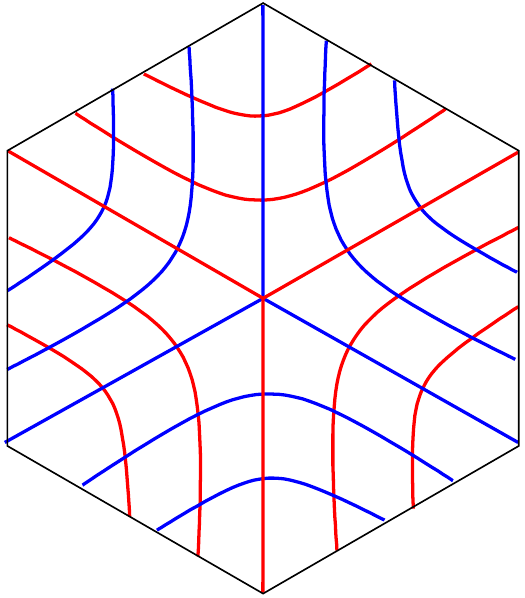}
\end{center}
\caption{A generalized box.}
\label{figSingPA}
\end{figure}

\begin{proof}[Proof of Theorem \ref{Teo A}]
Take a finite number of discs $D_1,\dots,D_n$ covering the surface
such that there is no bi-asymptotic sector in $D_j$ for all $j=1,\dots,n$.
Let $\delta>0$ be such that if the diameter of a set $X$ in the surface is smaller than $\delta$ then
$X$ is contained in some disc of the covering, and
fix $\sigma>0$ such that if $\diam(X)<\sigma$ then $\max\{\diam(f(X)),\diam(f^{-1}(X))\}<\delta$.
Take a finite covering $A_1,\dots,A_m$ of open sets
such that $\diam(A_i)<\sigma$ for all $i=1,\dots,m$.
Take $\alpha>0$ as a Lebesgue nomber of the covering $A_1,\dots,A_m$. Hence if $\dist(x,y)<\alpha$ then $x,y$ are in some $A_i$.
We assume, without loss, that each  $A_i$ has a local structure product or it is a generalized box,
as explained above.

Suppose by contradiction that $f$ is not $\alpha$-expansive and take $x\neq y$ such that
$$\dist(f^n(x),f^n(y))<\alpha,\quad \mbox{for all } n\in \Z .$$
Since each singularity is {\em dynamically isolated},
eventually changing $x,y$ for some of its iterates, we can assume that both $x$ and $y$ are contained, for some $k$,
in the same $A_k$ with product structure.

Suppose first the $x,y$ are neither in a stable arc of $A_k$ nor in an unstable arc of $A_k$.
Since $A_k$ has a product structure, the stable arc $s(x)$ at $x$ cuts the unstable arc $u(y)$ at $y$ in some point $z$ different from both $x$ and $y$.
The three points $x,y,z$ contradicts 2-expansiveness.

Now suppose that $x,y$ are in an unstable arc contained in $A_k$. 
Denote by $u$ this unstable arc. The diameter of $u$ is smaller than $\sigma$ and since it is unstable, there is a first $n_0>0$ such that
$\diam(f^{n_0}(u))>\sigma$ while the distance between $x'=f^{n_0}(x)$ and $y'=f^{n_0}(y)$ is less than $\alpha$. Let $u'=f^{n_0}(u)$, we have that $\diam(u')<\delta$. Then $u'$ is contained in some disc $D$ of the first considered
covering.
We have that $x'$ and $y'$ are in a generalized box $A'$ contained in $D$.
Therefore, $x'$ and $y'$ are joined by the unstable arc $u'\subset D$ and also, inside the $A'$, by an arc of type $stable-unstable$ (if there is in  $A'$ a local product structure) or by an arc of type $stable - unstable-stable$ if $A'$ is a generalized box as in Figure \ref{figSingPA}.
In any case it gives us a bi-asymptotic sector inside $D$ which is a contradiction.
\end{proof}

\section{Proof of Theorem \ref{Teo B}}
\label{sectionteoB}

In this section we prove Theorem \ref{Teo B}. All that what  follows is the description of an
example of a surface $2$-expansive homeomorphism with wandering points that is not expansive. 
This example were first considered by Alfonso Artigue, Joaquin Brum and Rafael Potrie, during a seminar course delivered by Jorge Lewowicz in Montevideo, Uruguay. 
 
The construction of this example  is based on the construction  of a quasi-Anosov diffeomorphism given in \cite{FR}.

Consider $S_1$ and $S_2$ two disjoint copies of the torus $\R^2/\Z^2$.
Let $f_i\colon S_i\to S_i$ be two diffeomorphisms such that:
\begin{itemize}
\item $f_1$ is a derived-from-Anosov (see for example \cite{Robinson} Section 7.8 for a construction of such a map),
\item $f_2$ is conjugated with $f_1^{-1}$,
\item $f_i$ has a fixed point $p_i$, $p_1$ is a source and $p_2$ is a sink,
\end{itemize}

Also assume that there are local charts $\varphi_i\colon D\to S_i$, $D=\{x\in\R^2:\|x\|\leq 2\}$,
such that
\begin{enumerate}
 \item $\varphi_i(0)=p_i$,
\item the pull-back of the stable (unstable) foliation by $\varphi_1$ ($\varphi_2$)
is the vertical (horizontal) foliation on $D$ and
\item $\varphi_1^{-1}\circ f^{-1}_1\circ \varphi_1(x)=\varphi_2^{-1}\circ f_2\circ \varphi_2(x)=x/4$
for all $x\in D$.
\end{enumerate}

Let $A$ be the annulus $\{x\in\R^2:1/2\leq \|x\|\leq 2\}$
and $\psi\colon\R^2\to \R^2$ the inversion $\psi(x)=x/\|x\|^2$.
Consider $\hat D$ the open disk $\{x\in\R^2:\|x\|<1/2\}$.
On $[S_1 \setminus \varphi_1(\hat D)]\cup [S_1\setminus \varphi_2(\hat D)]$
consider the equivalence relation generated by
\[
 \varphi_1(x)\sim \varphi_2\circ\psi (x)
\]
for all $x\in A$. Denote by $\overline x$ the equivalence class of $x$.
The surface
\[
S= \frac{[S_1 \setminus \varphi_1(\hat D)]\cup [S_1\setminus \varphi_2(\hat D)]}\sim
\]
is a bitorus with the quotient topology.
The stable and unstable foliations are illustrated in Figure \ref{fig}.
\begin{figure}[h]
\begin{center}
\includegraphics[scale=.9]{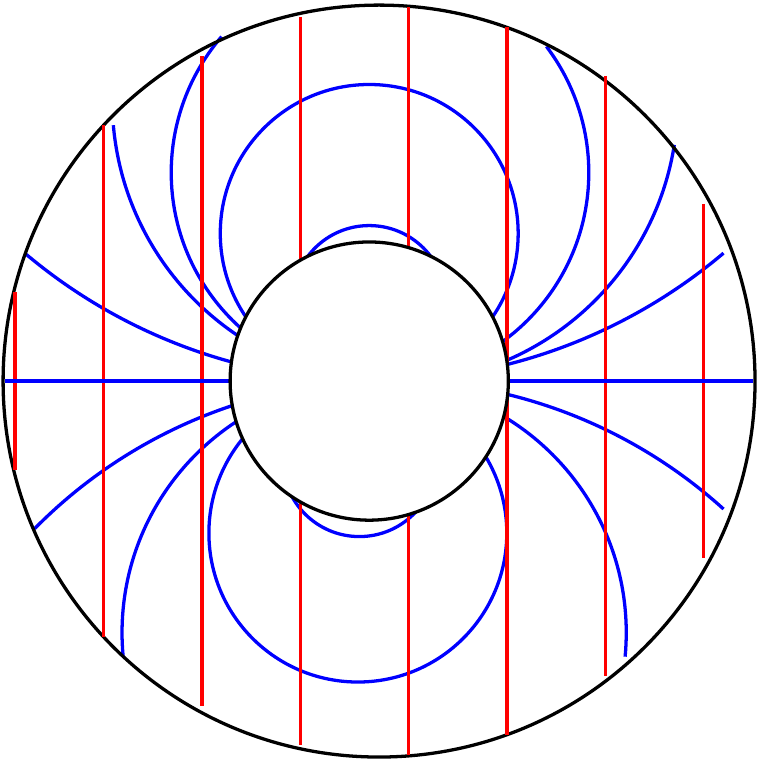}
\caption{Foliations in the annulus $\overline A$.
Blue lines represent the unstable foliation (after the inversion) and the red lines are the stable foliation.}
\label{fig}
\end{center}
\end{figure}

Consider the homeomorphism $f\colon S\to S$ defined by
\[
 f(\overline x)=\left\{
\begin{array}{ll}
\overline{f_1(x)} &\hbox{ if } x\in S_1 \setminus \varphi_1(\hat D)\\
\overline{f_2(x)} &\hbox{ if } x\in S_2 \setminus \varphi_2(D)\\
\end{array}
\right.
\]

\begin{Prop}
 The homeomorphism $f$ is 2-expansive but it is not expansive.
\end{Prop}

\begin{proof}
 It is not expansive because $\Omega f\neq S$.

To show that it is 2-expansive notice that
\begin{itemize}
 \item $\Omega f$ is expansive (because it is hyperbolic) and
\item $\Omega f$ is isolated, i.e. there is an open set $U$ such that $\Omega f=\cap_{n\in\Z} f^nU$.
\end{itemize}
So, it only rest to show that there is $\delta>0$ such that if $X\cap \Omega f=\emptyset$
and $\diam f^nX<\delta$ for all $n\in \Z$ then $|X|<3$.
Let $\overline A=\{\overline x:x\in \varphi_1(A)\}$.
Let $\delta>0$ be such that
$B_\delta (x)\subset f^{-1}\overline A\cup \overline A\cup f\overline A$
for all $x\in\overline A$.
By construction, we have that $W^s_\delta(x)\cap W^u_\delta(x)$ has at most two points if $x\in\overline A$.
Notice that for all $x\notin\Omega f$ there is $n\in\Z$ such that $f^nx\in \overline A$.
This finishes the proof.
\end{proof}

\noindent
{\em  M. J. Pacifico}:
Instituto de Matem\'atica,
Universidade Federal do Rio de Janeiro,
C. P. 68.530, CEP 21.945-970,
Rio de Janeiro, RJ, Brazil.

E-mail: pacifico@im.ufrj.br .\\

\noindent
{\em A. Artigue and J. Vieitez}:
Regional Norte, Universidad de la Republica,
Rivera 1350, CP 50000, Salto, Uruguay.

E-mail: artigue@unorte.edu.uy

E-mail: jvieitez@unorte.edu.uy.
\end{document}